\documentclass[graybox]{svmult}

\usepackage{type1cm}        
\usepackage{makeidx}         
\usepackage{graphicx}        
\usepackage{multicol}        
\usepackage[bottom]{footmisc}

\usepackage{newtxtext}       %
\usepackage[varvw]{newtxmath}       

\makeindex             

\newcommand{\cA}{\mathcal{A}}

\newcommand{\cF}{\mathcal{F}}

\newcommand{\cO}{\mathcal{O}}

\newcommand{\bC}{\mathbb{C}}

\newcommand{\bF}{\mathbb{F}}

\newcommand{\bK}{\mathbb{K}}

\newcommand{\bP}{\mathbb{P}}

\newcommand{\bR}{\mathbb{R}}

\newcommand{\be}{\begin{equation}}
\newcommand{\ee}{\end{equation}}
\newcommand{\bal}{\begin{align}}
\newcommand{\eal}{\end{align}}
\newcommand{\ba}{\begin{align*}}
\newcommand{\ea}{\end{align*}}
\newcommand{\bmx}{\begin{matrix}}
\newcommand{\emx}{\end{matrix}}
\newcommand{\bbmx}{\begin{bmatrix}}
\newcommand{\ebmx}{\end{bmatrix}}
\newcommand{\bpmx}{\begin{pmatrix}}
\newcommand{\epmx}{\end{pmatrix}}
\newcommand{\bvmx}{\begin{vmatrix}}
\newcommand{\evmx}{\end{vmatrix}}

\newcommand{\ol}{\overline}

\newcommand{\wt}{\widetilde}
\newcommand{\f}{\frac}
\newcommand{\df}{\dfrac}

\newcommand{\im}{\Rightarrow}

\newcommand{\inc}{\subseteq}

\newcommand{\setm}{\setminus}

\renewcommand{\Re}{\operatorname{Re}}

\newcommand{\la}{\lambda}

\newcommand{\eps}{\varepsilon}
  
\begin{document}

\title*{
Linearly Embedding Sparse Vectors from $\ell_2$ to~$\ell_1$
via Deterministic Dimension-Reducing Maps 
}
\titlerunning{Linearly Embedding Sparse Vectors from $\ell_2$ to~$\ell_1$}
\author{Simon Foucart}
\institute{Simon Foucart \at Texas A\&M University,  College Station\\\email{foucart@tamu.edu}
}
%
%
\maketitle

\abstract{This note is concerned with deterministic constructions of $m \times N$ matrices satisfying a restricted isometry property from $\ell_2$ to $\ell_1$ on $s$-sparse vectors.
Similarly to the standard ($\ell_2$ to $\ell_2$)
restricted isometry property,
such constructions can be found in the regime $m \asymp s^2$,
at least in theory.
With effectiveness of implementation in mind,
two simple constructions are presented in the less pleasing but still relevant regime $m \asymp s^4$.
The first one,
executing a Las Vegas strategy, is quasideterministic and applies in the real setting.
The second one, exploiting Golomb rulers,
is explicit and applies to the complex setting.
As a stepping stone, an explicit isometric embedding from $\ell_2^n(\bC)$ to $\ell_4^{cn^2}(\bC)$ is presented.
Finally, the extension of the problem from sparse vectors to low-rank matrices is raised as an open question.}

\section{Motivation from Sparse Vector Recovery}

Almost twenty years ago
\cite{CRT,Don},
the realization that high-dimensional but sparse vectors could be efficiently recovered from far fewer linear measurements than expected created a prolific field of research now known as compressive sensing (or compressed sensing).
To be specific,
vectors $x \in \bK^N$, $\bK \in \{\bR, \bC\}$,
are called $s$-sparse if
$$
\|x\|_0 := |{\rm supp}(x)| \le s,
\qquad \mbox{where }
{\rm supp}(x) := \{ j \in [1 : N]: x_j \not= 0\}.
$$
Such vectors can be recovered from compressive measurements $A x \in \bK^m$ with $m$ being of the order of $s \ln(N/s) \ll N$.
One refers to \cite{BookCS} for all the nitty-gritty details.
One simply mentions here that,
on the one hand,
the order $s \ln(N/s)$ cannot be lowered if one requires the recovery to be stable
and, 
one the other hand, 
that random measurement matrices $A \in \bK^{m \times N}$ with $m \asymp s \ln(N/s)$ fulfill, 
with high probability,
favorable properties that make $s$-sparse recovery possible.
Thus, there is an abundance of matrices suitable for compressive sensing in the optimal regime $m \asymp s \ln(N/s)$,
but somehow the mathematical community is unable to pinpoint a single one!

The most popular favorable property---the restricted isometry property (RIP),
introduced in \cite{CanTao}---
stipulates that the matrix $A \in \bK^{m \times N}$
should satisfy
\be
\label{SRIP}
(1-\delta) \|x\|_2^2 
\le \| A x\|_2^2  \le (1+\delta) \|x\|_2^2 
\qquad \mbox{for all $s$-sparse } x \in \bK^N. 
\ee
This standard version of the RIP is ubiquitous in ensuring the success of sparse recovery via a variety of reconstruction algorithms,
such as 
$\ell_1$-minimization (aka basis pursuit),
orthogonal matching pursuit (OMP),
compressive sampling matching pursuit (CoSaMP),
iterative hard thresholding (IHT),
hard thresholding pursuit (HTP),
to name but a few.
It can be interpreted as saying that the linear map $x \mapsto Ax$ provides an embedding from $\ell_2^N$ to $\ell_2^m$ with a distortion on $s$-sparse vectors equal to $\gamma  = (1+\delta)/(1-\delta) \ge 1$.
The latter can be made arbitrarily close to one by taking $\delta>0$ small enough.
However,
the distortion need not be close to one to enable sparse recovery:
a requirement $\gamma \le \gamma^*$ for a fixed threshold $\gamma^* \ge 1$ is enough.
Likewise,
the embedding need not map into $\ell_2^m$:
any $\ell_p^m$ with $0 < p \le 2$ will be convenient.
This note concentrates on the case $p=1$---interestingly, this version appeared in~\cite{Don}.
Thus,
instead of the standard version of the RIP, i.e., \eqref{SRIP},
one considers
an `$\ell_2$ to $\ell_1$' RIP
stipulating that $A \in \bK^{m \times N}$ should satisfy
\be
\label{RIP21}
\alpha \|x\|_2 
\le \| A x\|_1  \le \beta \|x\|_2 
\qquad \mbox{for all $s$-sparse } x \in \bK^N
\ee
with distortion on $s$-sparse vectors bounded by a fixed threshold,
say $\gamma = \beta/\alpha \le \gamma_*$.
I am an advocate of this alternative version, for several reasons:

\begin{itemize}
\item the theory of sparse recovery can be built from \eqref{RIP21},
and not only for basis pursuit as presented in \cite[Chapter 14]{BookDS},
but also for iterative hard thresholding, see \cite{FouLec};

\item the theory of one-bit compressive sensing can be built from \eqref{RIP21} as well,
as presented \cite[Chapter 17]{BookDS},
so long as the distortion can be made close to one,
which is the case if $A$ is a Gaussian matrix or a partial Gaussian circulant matrix~\cite{DJR};

\item Laplace matrices (more generally subexponential random matrices) satisfy~\eqref{RIP21} in the optimal regime $m \asymp s \ln(N/s)$, see \cite{FouLai},
while \eqref{SRIP} would only hold in the suboptimal regime 
$m \asymp s \ln^2(N/s)$;

\item any ensemble of random matrices $A \in \bK^{m \times N}$ yielding \eqref{SRIP} with exponentially small failure probability when $m \asymp s \ln(N/s)$
also yields \eqref{RIP21} with exponentially small failure probability when $m \asymp s \ln(N/s)$,
see \cite{LASSO} for the precise statement. 

\end{itemize}

This last point suggests that the `$\ell_2$ to $\ell_1$' RIP is easier to fulfill than the standard `$\ell_2$ to $\ell_2$' RIP,
so it is plausible that deterministic constructions of RIP matrices are more accessible through the `$\ell_2$ to $\ell_1$' avenue.
Since deterministic constructions exist for the standard RIP with $m \asymp s^{2-\eps}$ (see next section),
the same is expected to hold for the `$\ell_2$ to $\ell_1$' RIP.
This is indeed the case,
but through a construction that may be considered inadequate, 
because it combines two fairly theoretical results that are, in my view, not as explicit as hoped for.
The purpose this note is to exhibit a couple of simple constructions of matrices that fulfill the `$\ell_2$ to $\ell_1$' RIP in the regime $m \asymp  s^4$.
This is not the desired regime, for sure,
but the advantage here is the simplicity of the constructions.
This simplicity is validated by the few-lines {\sc matlab} implementation found in the associated  reproducible file
(available on the author's webpage).

The rest of this note is organized as follows.
Section \ref{SecDetK} discusses some known facts about deterministic embeddings.
Section \ref{SecEmbR} presents, in the case $\bK = \bR$,
the first deterministic construction of an embedding of $s$-sparse vectors from $\ell_2^N$ to $\ell_1^m$ with $m \asymp s^4$.
Section \ref{SecFullEmb} uncovers, in case $\bK = \bC$,
an explicit isometric embedding from $\ell_2^p$ to $\ell_4^m$ with $m \asymp p^2$, $p$ being a prime number.
Section \ref{SecEmbC} exploits this isometric embedding---or rather the argument leading to it---to present, in the case $\bK = \bC$,
a second deterministic, and actually explicit, construction 
of an embedding of $s$-sparse vectors from $\ell_2^N$ to $\ell_1^m$ with $m \asymp s^4$.
Section \ref{SecLowRk} briefly touches on the mostly uncharted territory of deterministic restricted isometry properties for low-rank matrices.
Finally, as an aside, Section \ref{SecLinks} recalls a connection between (almost) isometric embeddings from $\ell_2^n$ to $\ell_{2k}^N$ and (approximate) spherical designs.

\section{Known Deterministic Results}
\label{SecDetK}

The mathematical community's incapability to create nonrandom matrices fulfilling the standard RIP in the optimal regime is vexing.
Nonrandom procedures are stuck in the quadratic regime $m \asymp s^2$---in truth, $m \asymp s^2 {\rm polylog}(N)$.
There are several ways to reach this regime,
mostly based on the notion of coherence.
The coherence of a matrix $A \in \bK^{m \times N}$ with unit $\ell_2$-norm columns $a_1,\ldots,a_N \in \bK^m$
is defined by 
$$
\mu(A) := \max_{j \not= \ell} | \langle a_j, a_\ell \rangle|.
$$
Indeed, to ensure that $\delta_s(A)$,
the smallest $\delta \in (0,1)$ for which \eqref{SRIP} holds, obeys $\delta_s(A) < \delta_*$ as soon as $m \ge C s^2$,
it is sufficient make the coherence small as per 
$$
\mu(A) \le \f{c}{\sqrt{m}},
\qquad c:= \sqrt{C} \, \delta_*,
$$
by virtue of the inequality $\delta_s(A) < s \mu(A)$ (see e.g. \cite[Proposition 6.2]{BookCS}).
As examples of matrices with small coherence, let me mention
\begin{itemize}
\item an $m \times m^2$ matrix with columns formed by translations and modulations of the Alltop vector  (\cite{StrHea}, see also \cite[Proposition 5.13]{BookCS}):
its coherence is $\mu(A) = 1 /\sqrt{m}$;
\item when $p$ is prime and $d < p$,
a $p^2 \times p^{d+1}$ matrix with binary entries in $\{0,1/\sqrt{p}\}$ (\cite{DeV}, see also \cite[Theorem 18.5]{BookDS}):
its coherence is $\mu(A) \le d/p = d/\sqrt{m}$;
\item when $p$ is prime and $d < p$,
a $p \times p^{d+1}$ matrix with entries
\be
\label{EntriesWeil}
A_{k,f} = \f{1}{\sqrt{p}} \exp \left( i 2 \pi \f{k f(k)}{p} \right)
\ee
indexed by elements $k$ of $\bF_p$
and by polynomials $f$ over $\bF_p$ of degree at most $ d$:
its coherence is $\mu(A) \le d/\sqrt{p} = d/\sqrt{m}$.
The argument is simple but relies on a deep result known as Weil bound (see e.g.  \cite[Proposition 5.3.8]{NieWin}),
which says that,
if $f$ is a nonconstant polynomial over $\bF_p$,
then
$$
\bigg|
\sum_{k \in \bF_p} \exp \left( i 2 \pi \f{f(k)}{p} \right)
\bigg|
\le (\deg(f)-1) \sqrt{p}.
$$

\end{itemize}

There is a notable explicit construction that overcomes, albeit ever so slightly, 
the quadratic barrier.
Indeed, the article \cite{BDFKK}
uncovered an RIP matrix with $m \asymp s^{2-\eps}$ rows,
where $\eps > 0$ was tiny.
In a practitioner's mind,
this is viewed as an issue,
but in fact not as the most critical one:
one also has $m \ge N^{1-\eps}$,
making the matrix almost square and thus defeating 
the compressive sensing purpose of taking far fewer measurements than the high ambient dimension.
A similar issue occurs in~\cite{BMM}:
there, conditionally on a folklore conjecture in number theory,
it was shown that the quadratic barrier would be overcome by the Paley matrix fulfilling the `$\ell_2$ to $\ell_2$' RIP,  
but this matrix has $m = N/2$ rows.

Turning now to the deterministic `$\ell_2$ to $\ell_1$' RIP,
which has not been explored,
at least to the best of my knowledge,
one can follow the indirect strategy below:
\begin{itemize}
\item[(i)] \, consider a deterministic linear map $A': \ell_2^N \to \ell_2^n$ with constant distortion on $s$-sparse vectors for $n \asymp s^\eta$;
\item[(ii)] \, consider a deterministic linear map $A'': \ell_2^n \to \ell_1^m$ with constant distortion on arbitrary vectors for $m \asymp n^\theta$.
\end{itemize}
Then,
the deterministic map $A := A'' \circ A': \ell_2^N \to \ell_1^m$ has constant distortion on $s$-sparse vectors for $m \asymp s^{\eta \theta}$.
For~(i), one can take $\eta = 2 - \eps$ according to \cite{BDFKK}.
For~(ii), according to \cite{Ind} (or \cite{GLR},
which both rely on \cite{GUV}),
one can take $\theta = 1+\omega$ for any $\omega > 0$.
These choices yields an `$\ell_2$ to $\ell_1$' RIP in the regime $m \asymp s^{(2-\eps)(1+\omega)}$ for any $\omega >0$,
hence overcoming the quadratic barrier here, too.
This is quite compelling until the hands-on stage,
because the theoretical results underlying the argument are not easily implementable,
despite being deterministic 
(which often means constructible in polynomial time).

The general strategy is still valid, though,
and one can gain in explicitness by giving up on the smallest $\nu$ in $m \asymp s^\nu$.
This is the spirit of Section \ref{SecEmbC}, which obtains $\nu=4$ by way of taking $\eta = 2$ in~(i) and $\theta = 2$ in~(ii).
For the first step,
one can exploit any of the small-coherence matrices listed above.
For the second step,
one relies on the existence of deterministic constant-distortion linear embeddings from $\ell_2^n$ to $\ell_1^{c n^2}$.
Such an existence result is stated in \cite[Lemma 3.3]{LLR} by citing \cite{Ber} for a construction of $4$-wise independent families of vectors,
which I would qualify as deterministic but not explicit.
The article \cite{GLR} also claims without details that such an existence result can be extracted from \cite{Rud}.
Section \ref{SecEmbC} will actually provide such an explicit map embedding $\ell_2^n$ into $\ell_1^{c n^2}$.
I suspect this construction to be very close to what should have been extracted from \cite{Rud}.
It will yield an explicit embedding on sparse vectors taking place in the complex setting.
In the real setting, 
the construction presented next in Section \ref{SecEmbR} is more direct,
as it bypasses (i)-(ii),
but I would qualify it as quasideterministic rather than explicit.
Both constructions pertain to the regime $m \asymp s^4$.
This should definitely be improved to $m \asymp s^\nu$ with $\nu \le 2$,
but for the moment there is comfort in the simplicity of the constructions.

\section{Simple embedding of $s$-sparse vectors from $\ell_2^N(\bR)$ to $\ell_1^{c s^4}(\bR)$}
\label{SecEmbR}

The construction proposed in this section is
not explicit, but rather (quasi)deterministic, in the sense that it is outputted by the following Las Vegas algorithm:

\vspace{3mm}
\noindent
for $t=1,2,\ldots$
\begin{itemize}
\item draw $A \in \bR^{m \times N}$ populated with independent Rademacher random  variables,
\item
check if conditions (a)-(b) of Theorem \ref{Thm1} below are satisfied for $\kappa = \sqrt{8 \ln(N)}$,
\end{itemize}
and if they are, then return $A$.

The success of this procedure is guaranteed by the fact 
that conditions (a)-(b) certify that $A$ yields an embedding on $s$-sparse vectors from $\ell_2^N$ to $\ell_1^m$ (by Theorem~\ref{Thm1}), together with the fact that each draw of $A$ satisfies (a)-(b) for $\kappa = \sqrt{8 \ln(N)}$ with failure probability at most $1/3$,
so the failure probability after $T$ independent rounds is at most $(1/3)^T$.
This vanishing quantity is never exactly zero,
hence the procedure cannot technically be qualified as deterministic,
but quasideterministic sounds like a suitable designation. 
Note that the verification of (a)-(b) can be performed in polynomial time,
precisely in $\cO(N^4m)$ multiplications.

\begin{theorem}
\label{Thm1}
Let $A \in \bR^{m \times N}$ be populated with entries $A_{j,k} = \pm 1$.
Assume that
\begin{itemize}
\item[{\rm (a)}] 
\; $\displaystyle{ \bigg| \sum_{j=1}^m A_{j,k} A_{j,k'} \bigg| \le \kappa \sqrt{m}}\qquad \qquad \qquad $ for all distinct $ k, k' \in [1:N]$,
\item[{\rm (b)}] 
\; $\displaystyle{ \bigg| \sum_{j=1}^m A_{j,k} A_{j,k'} A_{j,\ell} A_{j,\ell'} \bigg| \le \kappa \sqrt{m}}\qquad \,$ for all distinct $k, k',\ell, \ell' \in [1:N]$.
\end{itemize}
Then, for any $\delta \in (0,1)$, the linear map $A: \ell_2^N \to \ell_1^m$ has the property that
$$
\alpha m \|x\|_2 
\le \| A x\|_1  \le \beta m \|x\|_2 
\qquad \quad \mbox{for all $s$-sparse } x \in \bR^N,
$$
with distortion $ \gamma = \beta/\alpha \le \sqrt{3} \big( (1+\delta)/(1-\delta) \big)^{3/2}$
as soon as $m \ge \kappa^2 \delta^{-2} s^4$.
\end{theorem}

Before justifying this theorem,
it is worth observing that its assumptions (a)-(b) are indeed fulfilled when $A \in \bR^{m \times N}$ is a Rademacher random matrix and $\kappa = \sqrt{8 \ln(N)}$.
Fixing distinct $k, k'$ in $[1:N]$,
the $A_{j,k} A_{j,k'}$, $j \in [1:m]$,
are independent Rademacher variables,
so by Hoeffding inequality (see e.g. \cite[Corollary 8.8]{BookCS})
$$
\bP \left[  \bigg| \sum_{i=1}^m A_{i,j_1} A_{i,j_2}  \bigg| \ge \kappa \sqrt{m}
\right] \le 2 \exp \left(- \f{\kappa^2}{2} \right). 
$$
Hence, by a union bound, one derives
$$
\bP[ \mbox{(a) fails}]
\le \binom{N}{2} \, 2 \exp\left(- \f{\kappa^2}{2}\right) \le N^2 \exp\left(- \f{\kappa^2}{2} \right) = \f{1}{N^2}.
$$
Likewise,
fixing distinct $k,k',\ell,\ell' \in [1:N]$,
the $A_{j,k}A_{j,k'} A_{j,\ell} A_{i,\ell'}$,
$j \in [1:m]$,
are independent Rademacher variables,
so Hoeffding inequality followed by a union bound once again yields
$$
\bP[ \mbox{(b) fails} ]
\le \binom{N}{4} \, 2 \exp\left(- \f{\kappa^2}{2}\right) \le \f{N^4}{12} \exp\left(- \f{\kappa^2}{2} \right) = \f{1}{12}.
$$
Consequently, (a) and (b) are indeed both fulfilled with failure probability bounded, for $N \ge 2$, as
$$
\bP [\mbox{(a) or (b) fail} ]
\le \f{1}{N^2} + \f{1}{12} \le \f{1}{3}.
$$

Before turning the attention to the proof of Theorem \ref{Thm1},
it is also worth isolating
two key ingredients as separate lemmas,
one being a technical calculation to be reused later 
and the other one being a way to estimate the $\ell_1$-norm from below via the $\ell_2$- and $\ell_4$-norms.

\begin{lemma}
\label{LemB}
Let $B \in \bC^{q \times r}$ with $|B_{j,k}| = 1$ for all $j \in [1:q]$ and $k \in [1:r]$.
Then, for any $x \in \bC^r$, one has
\begin{align}
\label{IdenL2}
\|Bx\|_2^2
& = q \|x\|_2^2 + \sum_{1 \le k \not= k' \le r} \bigg( \sum_{j=1}^q \ol{B_{j,k}} B_{j,k'}\bigg) \ol{x_k} x_{k'} ,\\
\label{IdenL4}
\|Bx\|_4^4
& = 2 \|x\|_2^2 \|Bx\|_2^2 - q \|x\|_4^4 + \Sigma_1\\
\label{IdenL4bis}
& = 2 \|x\|_2^2 \|Bx\|_2^2 - q \|x\|_4^4 
+  \sum_{k \not= k'} \bigg( \sum_{j=1}^q \ol{B_{j,k}}^2 B_{j,k'}^2   \bigg) \ol{x_k}^2 x_{k'}^2
+ \Sigma_2,
\end{align}
where the quantities $\Sigma_1$ and $\Sigma_2$ are given by
\begin{align*}
\Sigma_1 & := \sum_{(k\not= k') \not= (\ell \not= \ell')} 
\bigg( \sum_{j=1}^q \ol{B_{j,k}} B_{j,k'} B_{j,\ell} \ol{B_{j,\ell'}}  \bigg) \ol{x_k} x_{k'} x_\ell \ol{x_{\ell'}},\\
\Sigma_2 & := \sum_{\substack{(k\not= k') \not= (\ell \not= \ell') \\ (k\not= k') \not= (\ell' \not= \ell)}} 
\bigg( \sum_{j=1}^q \ol{B_{j,k}} B_{j,k'} B_{j,\ell} \ol{B_{j,\ell'}}  \bigg) \ol{x_k} x_{k'} x_\ell \ol{x_{\ell'}}.
\end{align*}
\end{lemma}

\begin{proof}
For the first identity, 
one writes
\begin{align*}
\|B x\|_2^2 & =
\sum_{j=1}^q \ol{(Bx)_j} (Bx)_j
= \sum_{j=1}^q \bigg( \sum_{k=1}^r \ol{B_{j,k}} \ol{x_k} \bigg)
\bigg( \sum_{k'=1}^r B_{j,k'} x_{k'} \bigg)\\
& = \sum_{j=1}^q \bigg(
\sum_{k=1}^r \ol{B_{j,k}} \ol{x_k} B_{j,k} x_{k}
+ \sum_{k \not= k'} \ol{B_{j,k}} \ol{x_k} B_{j,k'} x_{k'}
\bigg)\\
& = \sum_{j=1}^q \sum_{k=1}^r |B_{j,k}|^2 |x_k|^2
+ \sum_{j=1}^q \sum_{k \not= k'} \ol{B_{j,k}} B_{j,k'} \ol{x_k} x_{k'},
\end{align*}
which, by virtue of $|B_{j,k}| = 1$,
reduces to \eqref{IdenL2} after changing the order of summation.

For the second identity,
one writes
\begin{align}
\nonumber
\|Bx\|_4^4 & = \sum_{j=1}^q \bigg[ \ol{(Bx)_j} (Bx)_j \bigg]^2
= \sum_{j=1}^q \bigg[ \bigg( \sum_{k=1}^r \ol{B_{j,k}} \ol{x_k} \bigg)
\bigg( \sum_{k'=1}^r B_{j,k'} x_{k'} \bigg) \bigg]^2\\
\nonumber
& = \sum_{j=1}^q \bigg[ \sum_{k=1}^r |B_{j,k}|^2 |x_k|^2 + S_j  \bigg]^2
\qquad \qquad \quad
\mbox{where } S_j :=  \sum_{k \not = k'} \ol{B_{j,k}} B_{j,k'} \ol{x_k} x_{k'}
\\
\label{PrepB4}
& = \sum_{j=1}^q \left[ \|x\|_2^2 + S_j \right]^2
= q \|x\|_2^4 + 2 \|x\|_2^2 \sum_{j=1}^q S_j + \sum_{j=1}^q S_j^2.
\end{align}
Taking \eqref{IdenL2} into account, one notices that
\be
\label{SumSj}
\sum_{j=1}^q S_j = \|Bx\|_2^2 - q \|x\|_2^2.
\ee
Next, exploiting the fact that $S_j \in \bR$ and separating the cases $(k\not= k') = (\ell \not= \ell')$ and $(k\not= k') \not= (\ell \not= \ell')$,
one obtains
\begin{align*}
S_j^2 & =  S_j \ol{S_j}\\
& = \sum_{k\not= k'} 
| B_{j,k}|^2 | B_{j,k'}|^2 |x_{k}|^2 |x_{k'}|^2 
+ \sum_{(k\not= k') \not= (\ell \not= \ell')} 
\ol{B_{j,k}} B_{j,k'} B_{j,\ell} \ol{B_{j,\ell'}}
\ol{x_k} x_{k'} x_\ell \ol{x_{\ell'}},
\end{align*}
from where it follows that
\be
\label{SumSj^2}
\sum_{j=1}^q S_j^2
=  \sum_{j=1}^q \bigg( \sum_{k,k'} |x_k|^2 |x_{k'}|^2 - \sum_k |x_k|^4 \bigg) + 
\Sigma_1
= q \left( \|x\|_2^4 - \|x\|_4^4 \right) + \Sigma_1.
\ee
It remains to substitute \eqref{SumSj} and \eqref{SumSj^2} into \eqref{PrepB4} to arrive at \eqref{IdenL4}.

For the last identity, one simply writes
\begin{align*}
\Sigma_1 & = 
\Bigg\{
\sum_{\substack{(k\not= k') \not= (\ell \not= \ell') \\ (k\not= k') = (\ell' \not= \ell)}} 
+ \sum_{\substack{(k\not= k') \not= (\ell \not= \ell') \\ (k\not= k') \not= (\ell' \not= \ell)}}
\Bigg\}
\bigg( \sum_{j=1}^q \ol{B_{j,k}} B_{j,k'} B_{j,\ell} \ol{B_{j,\ell'}}  \bigg) \ol{x_k} x_{k'} x_\ell \ol{x_{\ell'}}\\
& = \sum_{k \not= k'} \bigg( \sum_{j=1}^q \ol{B_{j,k}}^2 B_{j,k'}^2   \bigg) \ol{x_k}^2 x_{k'}^2 
+ \Sigma_2,
\end{align*}
and substituting the latter into \eqref{IdenL4}
directly leads to \eqref{IdenL4bis}. 
\end{proof}

As for the lower estimate of the $\ell_1$-norm,
this is achieved via an upper estimate of the $\ell_4$-norm,
as stated below.

\begin{lemma}
\label{LemL1L4}
For any $y \in \bC^m$, 
$$
\|y\|_1 \ge \f{\|y\|_2^3}{\|y\|_4^2}.
$$
\end{lemma}

\begin{proof}
This is H\"older inequality in disguise.
Namely, one can easily rearrange
$$
\|y\|_2^2 
= \sum_{i=1}^m |y_i|^{1 \times \f{2}{3} + 4 \times \f{1}{3}}
\le \left( \sum_{i=1}^m |y_i|^1 \right)^{\f{2}{3}}
\left( \sum_{i=1}^m |y_i|^4 \right)^{\f{1}{3}}
= \|y\|_1^{\f{2}{3}} \|y\|_4^{\f{4}{3}}
$$
into the announced inequality.
\end{proof}

One is now ready to justify the main result of this section.

\begin{proof}[of Theorem \ref{Thm1}]
Let an $s$-sparse vector $x \in \bR^N$ be fixed throughout the proof.
Applying Lemma \ref{LemB} with the real matrix $A \in \bR^{m \times N}$ taking the role of the complex matrix $B \in \bC^{q \times r}$,
one first observes,
from \eqref{IdenL2},
that
$$
 \|Ax\|_2^2  - m \|x\|_2^2 
= \sum_{ k \not= k' } \bigg( \sum_{j=1}^m A_{j,k} A_{j,k'}\bigg) x_k x_{k'}.
$$
Taking the bound (a) into consideration,
as well as the sparsity of $x$,
yields
\begin{align*}
\left| \|Ax\|_2^2  - m \|x\|_2^2 \right|
\le \sum_{ k \not= k' } \kappa \sqrt{m} |x_k| |x_{k'}|
\le \kappa \sqrt{m} \|x\|_1^2 
\le \f{\kappa s}{\sqrt{m}} m \|x\|_2^2. 
\end{align*}
For $m \ge \kappa^2 \delta^{-2} s^2$
(which is the case since $m \ge \kappa^2 \delta^{-2} s^4$), this implies
\be
\label{IntermRIP22}
\left( 1-\delta \right) m \|x\|_2^2 
\le \|A x\|_2^2 \le
\left( 1+\delta \right) m \|x\|_2^2 .
\ee
As a side note, this is the standard RIP for the renormalized matrix $A/\sqrt{m}$.
It was derived solely from (a),
which is nothing but a coherence assumption for this matrix.
Now, using the standard comparison of the $\ell_1$-norm and $\ell_2$-norm,
one arrives at
\be 
\label{UB}
\|A x\|_1 \le \sqrt{m} \, \|Ax\|_2
\le ( 1+\delta )^{1/2} \, m \|x\|_2.
\ee 
Next, using \eqref{IdenL4bis} in Lemma \ref{LemB}
while taking into account that $(k \not= k') \not= (\ell \not= \ell')$ and $(k \not= k') \not= (\ell' \not= \ell)$ means that $k,k',\ell,\ell'$ are all distinct,
and using assumption~(b) as well,
one can write
\begin{align*}
\|Ax\|_4^4 & \le 2 \|x\|_2^2 \|Ax\|_2^2 + m \|x\|_2^4
+ \sum_{\substack{k,k',\ell,\ell'\\ {\rm all \, distinct}}}
\bigg| \sum_{j=1}^m A_{j,k} A_{j,k'} A_{j,\ell} A_{j,\ell'} \bigg|
|x_k| |x_{k'}| |x_\ell| |x_{\ell'}|\\
& \le (2(1+\delta) + 1) m \|x\|_2^4
+ \kappa \sqrt{m} \|x\|_1^4
\le (3+2\delta ) m \|x\|_2^4
+ \kappa \sqrt{m} s^2 \|x\|_2^4\\
& = \left( 3+2 \delta + \f{\kappa s^2}{\sqrt{m}} \right) m \|x\|_2^4 \le (3+3\delta) m \|x\|_2^4,
\end{align*}
where the last step exploited $m \ge \kappa^2 \delta^{-2} s^4$.
This upper bound on $\|Ax\|_4$,
combined with the lower bound on $\|Ax\|_2$ from \eqref{IntermRIP22},
ensures,
according to Lemma \ref{LemL1L4}, that
\be
\label{LB}
\|Ax\|_1 \ge \f{\big( \|Ax\|_2^2 \big)^{3/2}}{\big( \|Ax\|_4^4 \big)^{1/2}}
\ge \f{\big((1-\delta) m \|x\|_2^2 \big)^{3/2}}{\big(   3(1+\delta) m \|x\|_2^4 \big)^{1/2}} = \left( \f{(1-\delta)^3}{3(1+\delta)} \right)^{1/2} m \|x\|_2.
\ee
The estimates \eqref{UB} and \eqref{LB} together establish the result,
noting in particular the expression of the distortion.
\end{proof}

\begin{remark}
It is unclear if the above argument can be refined to improve the exponent $ \nu = 4$ for the 
regime $m \asymp s^\nu$ where the `$\ell_2$ to $\ell_1$' RIP provably holds.
One point is certain, however:
for a matrix $A \in \bK^{m \times N}$ whose entries all have modulus/absolute value equal to one,
one cannot beat $\nu = 2$
if one estimates the $\ell_1$-norm from below via an upper bound on the $\ell_4$-norm of the form $\|Ax\|_4^4 \le C m \|x\|_2^4$ for all $s$-sparse vectors $x \in \bK^N$.
Indeed, for a fixed $i \in [1:m]$, if $x = A_{S,i}$ represents the $i$th rows of $A$ restricted to a set $S \inc [1:N]$ of size $s$,
then
$$
C m s^2 = C m \|x\|_2^4
\ge \|A x\|_4^4 = \sum_{j=1}^m |\langle A_{:,j}, A_{S,i} \rangle|^4 \ge |\langle A_{:,i}, A_{S,i} \rangle|^4
= s^4,
$$
which forces $m \ge C^{-1} s^2$.
\end{remark}

\section{Explicit isometric embedding from $\ell_2^p(\bC)$ into $\ell_4^{c p^2}(\bC)$}
\label{SecFullEmb}

As mentioned earlier,
deterministic linear embeddings from $\ell_2^n$ to $\ell_1^{c n^2}$ with constant distortion have been claimed to exist in different places without full details,
e.g. stating that a construction can be extracted from \cite{Rud}.
In the latter, a central role was played by Sidon sets,
which are the same as Golomb rulers.
In what follows, I~propose a Golomb-ruler argument which likely coincides with what should have been extracted from \cite{Rud}.
Taking a detour via the $\ell_4$-norm,
the argument actually uncovers a deterministic embedding from $\ell_2^p$ ($p$ being prime) to $\ell_4^{ c p^2}$ whose distortion is exactly equal to one.
This isometric embedding is singled out in this section.
The main ingredient,
due to \cite{ErdTur}, is reproduced here for completeness.

\begin{lemma}
\label{LemGolomb}
For a prime number $p \ge 3$, the integers
$$
g(k) = 2 p k + \big( k^2 \big)_p,
\qquad k \in [0:p-1],
$$
where $\big( k^2 \big)_p$ denotes the integer $t \in [0:p-1]$ such that $k^2 \equiv t \mod p$, form a Golomb ruler, in the sense that
$$
g(k) - g(k') \not= g(\ell) - g(\ell')
\qquad \mbox{whenever } (k \not= k') \not= (\ell \not= \ell').
$$
\end{lemma}

\begin{proof}
Let $(k \not= k') \not= (\ell \not= \ell')$
and suppose $g(k)-g(k') = g(\ell)-g(\ell')$.
Writing $k' = k+ \sigma$ and $\ell' = \ell + \tau$
for some $\sigma,\tau \in [-p+1:p-1] \setm \{0\}$,
this reads 
$2pk + \big( k^2 \big)_p - 2p (k+\sigma) - \big( (k+\sigma)^2 \big)_p
= 2p \ell  + \big( \ell^2 \big)_p - 2p (\ell+\tau) - \big( (\ell+\tau)^2 \big)_p$,
i.e.,
\begin{equation}
\label{Modulo}
\big( k^2 \big)_p - \big( (k+\sigma)^2 \big)_p
- \big( \ell^2 \big)_p + \big( (\ell+\tau)^2 \big)_p
= 2 p (\sigma - \tau).
\end{equation}
Since the left-hand side, in absolute value, is at most $2(p-1)<p$,
the right-hand side must be zero,
so that $\sigma = \tau$.
Taking $\big( (k+\sigma)^2 \big)_p \equiv \big( k^2 \big)_p + \big( 2 \sigma k \big)_p+ \big( \sigma^2 \big)_p \mod p$ into account,
as well as $\big( (\ell+\tau)^2 \big)_p = \big( (\ell+\sigma)^2 \big)_p \equiv \big( \ell^2 \big)_p + \big( 2 \sigma \ell \big)_p+ \big( \sigma^2 \big)_p \mod p$,
looking at \eqref{Modulo} modulo $p$ yields
$$
\big( 2 \sigma \ell \big)_p - \big( 2 \sigma k \big)_p \equiv 0 \mod p,
\qquad \mbox{i.e.,} \qquad 2 \sigma (\ell - k) \equiv  0 \mod p.
$$
But since $2 \not\equiv 0 $ and $\sigma \not\equiv 0 \mod p$,
one deduces that $\ell - k \equiv 0 \mod p$
and hence that $k=\ell$.
In view of $k'=k+\sigma$ and $\ell' = \ell + \tau$ with $\sigma = \tau$, it follows that $(k \not= k') = (\ell \not= \ell')$.
This leads to a contradiction,
showing that $g$ indeed generates a Golomb ruler. 
\end{proof}

The coveted isometric embedding from $\ell_2^p$ to $\ell_4^{c p^2}$ will be obtained by combining Lemma \ref{LemB} with Lemma \ref{LemGolomb}.
Before that, it is worth pausing to remark that $g$ maps $[0:p-1]$
into $[0:q-1]$, where $q=3p(p-1)+1$ is quadratic in $p$,
and that such a quadratic order is optimal.
Indeed, for any Golomb ruler $g$ from $[0:p-1]$
into $[0:q-1]$, all the distinct $g(k)-g(k')$ indexed by ordered pairs $(k \not= k')$ are contained in $[-q+1:q-1]$, 
and as such $2q-1 \ge p(p-1)$.

\begin{theorem}
\label{ThmIsoEmb}
For a prime number $p \ge 3$, let $m=6p^2-6p+1$.
Consider the matrix $A'' \in \bC^{m \times p}$ with entries
\be
\label{PreMat}
A''_{j,k} = \exp \left( i 2 \pi \f{j g(k)}{m} \right),
\quad j \in [0:m-1], \; k \in [0:p-1].
\ee
Then the matrix $M \in \bC^{(m+p) \times p}$ defined as
\be
\label{DefA''}
M := \bbmx \df{1}{(2m)^{1/4}} A'' \\ \hline \df{1}{2^{1/4}} I_p \ebmx
\ee
provides an isometric embedding from $\ell_2^p(\bC)$ into $\ell_2^{m+p}(\bC)$, i.e.,
$$
\|M x\|_4 = \|x\|_2
\qquad \mbox{for all } x \in \bC^p.
$$
\end{theorem}

\begin{proof}
Let a vector $x \in \bC^p$ be fixed throughout the proof.
For any $k \not= k'$,
one observes that
$$
\sum_{j=0}^{m-1} \ol{A''_{j,k}} A''_{j,k'}
= \sum_{j=0}^{m-1} \exp \left( -i 2 \pi \f{j(g(k)-g(k'))}{m} \right) = 0,
$$
owing to $g(k)-g(k') \in [-q+1 : q-1] \inc [-m+1:m-1]$
being nonzero 
(otherwise, if $g(k)=g(k')$,
choosing $k'' \not\in \{k,k'\}$ would yield $g(k)-g(k'') = g(k')-g(k'')$).
According to Lemma \ref{LemB}, and specifically to \eqref{IdenL2}, one therefore has
$$
\|A'' x\|_2^2 = m \|x\|_2^2.
$$
Next, for any $(k \not= k') \not= (\ell \not= \ell')$,
one observes that
$$
\sum_{j=0}^{m-1} \ol{A''_{j,k}} A''_{j,k'} A''_{j,\ell} \ol{A''_{j,\ell'}}
=
\exp \left( -i 2 \pi \f{j\big( (g(k)-g(k')) - (g(\ell)-g(\ell')) \big)}{m} \right) =0,
$$
owing to $(g(k)-g(k')) - (g(\ell)-g(\ell')) \in [-2q+2:2q-2] = [-m+1:m-1]$ being nonzero.
According to Lemma \ref{LemB} again,
and specifically to \eqref{IdenL4}, one obtains 
$$
\|A'' x\|_4^4 = 2 \|x\|_2^2 \|A'' x\|_2^2 - m \|x\|_4^4 = 
2 m \|x\|_2^4 - m \|x\|_4^4.
$$
From here, it easily follows that
$$
\|M x\|_4^4 = \f{1}{2m} \|A'' x\|_4^4 + \f{1}{2} \|x\|_4^4
= \|x\|_2^4,
$$
which is the desired result.
\end{proof}

\begin{remark}
\label{RkEmb}
Isometric embeddings from $\ell_2$ to $\ell_4$ can alternatively be viewed through the lenses of spherical designs and of tensors.
In order not to be diverted from the main goal,
this connection will be brought forward much later,  in Section \ref{SecLinks}.
\end{remark}

\section{Explicit embedding of $s$-sparse vectors from $\ell_2^N(\bC)$ to $\ell_1^{c s^4}(\bC)$}
\label{SecEmbC}

Based on Theorem \ref{ThmIsoEmb} in the previous section, a linear embedding from $\ell_2^p(\bC)$ to $\ell_1^{c p^2}(\bC)$ with constant distortion can easily be generated.
However,
instead of using the matrix $M$ from \eqref{DefA''},
which provided an isometric embedding from $\ell_2^p(\bC)$ to $\ell_4^{c p^2}(\bC)$,
the matrix $A''$ from \eqref{PreMat} is preferred, as it leads to a nicer expression for the distortion.
Recall that the proof of Theorem \ref{ThmIsoEmb} revealed that,
for any $x \in \bC^p$,
$$
\|A''x\|_2 = \sqrt{m} \|x\|_2
\qquad \mbox{and} \qquad
\|A''x\|_4 \le (2 m )^{1/4} \|x\|_2.
$$

\begin{theorem}
\label{ThmEmbL1L2}
For a prime number $p \ge 3$, let $m = 6p^2-6p+1$.
The matrix $A'' \in \bC^{m \times p}$ defined in \eqref{PreMat} provides an embedding from $\ell_2^p(\bC)$ into $\ell_1^m(\bC)$ with distortion at most~$\sqrt{2}$. Precisely, one has 
$$
\f{m}{\sqrt{2}} \|x\|_2 \le \|A'' x\|_1 \le m \, \|x\|_2
\qquad 
\mbox{for all } x \in \bC^p.
$$
\end{theorem}

\begin{proof}
Let a vector $x \in \bC^p$ be fixed throughout this short proof.
On the one hand, comparing $\ell_1$- and $\ell_2$-norms yields
$$
\|A''x\|_1 \le \sqrt{m} \|A''x\|_2 = m \|x\|_2.
$$
On the other hand, according to Lemma \ref{LemL1L4},
one has
$$
\|A''x\|_1 \ge \f{\|A''x\|_2^3}{\|A'' x\|_4^2} \ge \f{\sqrt{m}^3 \|x\|_2^3}{\sqrt{2m} \|x\|_2^2}
= \f{m}{\sqrt{2}} \|x\|_2.
$$
These two inequalities together justify the announced embedding.
\end{proof}

It is now time for the main result of this section,
namely the awaited explicit linear embedding of $s$-sparse vectors from $\ell_2^N(\bC)$ to $\ell_1^{c s^4}(\bC)$.
It repeats the general strategy (i)-(ii) outlined in Section \ref{SecDetK}.

\begin{theorem}
Given integers $N, s \ge 1$ with $N \gg s^4$, 
let $p \ge 3$ be a prime number between $9s^2 \big\lceil \ln^2(N) \big\rceil$ and $18 s^2 \big\lceil \ln^2(N) \big\rceil$
and let $m = 6 p^2 - 6p + 1 \asymp s^4 \ln^2(N)$.
Then the matrix $A \in \bC^{m \times N}$
indexed by $[0:m-1]$ and by an arbitrary $N$-set $\cF$ of polynomials over $\bF_p$ of degree at most $d = \lceil \ln(N/p)/\ln(p) \rceil$ and with entries 
\begin{equation}
\label{ExpAjf}
A_{k,f} = \f{1}{\sqrt{p}} \hspace{-1mm}
\sum_{k=0}^{p-1} \exp\left( i 2 \pi \left( j \f{2pk + (k^2)_p}{m} + \f{k f(k)}{p} \right) \right),
\;  k \in [0\colon \hspace{-1mm} m-1], \, f \in \cF,
\end{equation}
provides an explicit embedding from $\ell_2^N(\bC)$ into $\ell_1^{m}(\bC)$ with distortion on $s$-sparse vectors at most $2$,
namely
\be
\label{MainRIP12}
\f{m}{\sqrt{3}} \|x\|_2 \le \|A x\|_1 \le \f{2 m}{\sqrt{3} } \|x\|_2
\qquad \mbox{for all $s$-sparse }x \in \bC^{N}.
\ee
\end{theorem} 

\begin{proof}
Picking a prime number $p$ between $9s^2 \big\lceil \ln^2(N) \big\rceil$ and $18 s^2 \big\lceil \ln^2(N) \big\rceil$ is possible by Bertrand postulate.
One considers the smallest integer $d$ such that $p^{d+1} \ge N$, i.e., $d = \lceil \ln(N/p)/\ln(p) \rceil$.
Note that $1 < p < d$ since 
$p^2 \lesssim s^4 {\rm polylog}(N) < N$
and $p^p \ge e^p \ge e^{\ln(N)} = N$.
From $p^d < N$, one also deduces that $d \le d \ln(p) < \ln(N)$.
Let then $A' \in \bC^{p \times p^{d+1}}$ be the `Weil' matrix with entries defined in \eqref{EntriesWeil},  recalling that
$\delta_s(A') < s \mu(A) < s d/\sqrt{p} < s \ln(N)/(3 s \ln(N)) = 1/3$,
so that
$$
\|A' x\|_2^2
\left\{
\bmx
\le (1+\delta_s(A')) \|x\|_2^2 \le \df{4}{3} \|x\|_2^2 \\
\\
\ge (1-\delta_s(A')) \|x\|_2^2 \ge \df{2}{3} \|x\|_2^2
\emx
\right.
\qquad \mbox{for all $s$-sparse } x \in \bC^{p^{d+1}}.
$$
Let also $A'' \in \bC^{m \times p}$
be the `Golomb' matrix with entries defined in \eqref{PreMat},
recalling that
$$
\|A'' x\|_1 
\left\{
\bmx
\le m \|x\|_2 \\ 
\\
\ge \df{m}{\sqrt{2}} \|x\|_2
\emx
\right.
\qquad \mbox{for all } x \in \bC^{p}.
$$
As a result, for the matrix $\wt{A} := A'' \times A' \in \bC^{m \times p^{d+1}}$ and for any $s$-sparse $x \in \bC^{p^{d+1}}$,
$$
\|\wt{A} x\|_1 = \| A''(A' x)\|_1
\left\{
\bmx
\le m \|A'x\|_2 \le m \sqrt{\df{4}{3}} \|x\|_2 = \df{2m}{\sqrt{3}} \|x\|_2,\quad \;\\
\ge \df{m}{\sqrt{2}} \|A'x\|_2 \ge \df{m}{\sqrt{2}} \sqrt{\df{2}{3}} \|x\|_2
= \df{m}{\sqrt{3}} \|x\|_2.
\emx
\right.
$$
This strongly resembles the desired inequalities \eqref{MainRIP12},
expect that the matrix $\wt{A}$ has $p^{d+1}$ columns,
while the matrix $A$ should have $N$ columns.
But one can simply remove $p^{d+1}-N$ arbitrary columns from $\wt{A} \in \bC^{m \times p^{d+1}}$ to create a matrix $ A \in \bC^{m \times N}$ satisfying \eqref{MainRIP12}.
The entries of this matrix, indexed by $j$  in $[0:m-1]$ and by $f$ in a set $\cF$ of polynomials over $\bF_p$ of degree at most $d$ with size $N$, are given by
$$
A_{j,f} = \sum_{k=0}^{p-1} A''_{j,k} A'_{k,f}
= \sum_{k=0}^{p-1} \exp \left( i 2 \pi \f{j g(k)}{m} \right) \f{1}{\sqrt{p}} \exp \left( i 2 \pi \f{k f(k)}{p} \right),
$$
which reduces to the expression announced in \eqref{ExpAjf}.
\end{proof}

\section{Outlook into low-rank recovery}
\label{SecLowRk}

The theory of compressive sensing also deals with objects of 
nominally high but intrinsically low dimension beyond $s$-sparse vectors $x \in \bK^N$,
prototypically with matrices $X \in \bK^{n \times n}$ of rank at most $r$.
In this scenario, too,
one can recover such objects from their compressive measurements $\cA(X) \in \bK^m$ with $m$ being of the order of $n r \ll n^2$.
This coup can be achieved (see \cite{RFP}) by nuclear norm minimization  when the linear map $\cA: \bK^{n \times n} \to \bK^m$ satisfies an RIP of the form
\be
\label{RRIP}
(1-\delta) \|X\|_F^2 \le \|\cA(X)\|_2^2 \le (1+\delta) \|X\|_F^2
\qquad \mbox{for all rank-$r$ } X \in \bK^{n \times n}. 
\ee
For $m \asymp n r$, 
this RIP is fulfilled when $A_1,\ldots,A_m \in \bK^{n \times n}$ in $\cA(Z)_i = \langle A_i, X \rangle_F$ are independent  random matrices populated with independent properly normalized gaussian entries, see \cite{CanPla}.
But, as in the vector case, no deterministic linear map $\cA: \bK^{n \times n} \to \bK^m$ is known in this optimal regime $m \asymp n r$.
Worst, I am not aware of simple deterministic constructions overcoming the trivial regime $m \asymp n^2$
and I~have not seen an effective analog of the notion of coherence. 

Interestingly, the theory of low-rank recovery can also be built from a modification  of the RIP featuring the $\ell_1$-norm as the inner norm, see \cite[Chapter 16]{BookDS}, namely from
\be
\label{MRRIP}
(1-\delta) \|X\|_F \le \|\cA(X)\|_1 \le (1+\delta) \|X\|_F
\qquad \mbox{for all rank-$r$ } X \in \bK^{n \times n}. 
\ee
This alternative version seems even more relevant in the present scenario.
Indeed, when $m \asymp n r$, the rank-one measurements
$$
\cA(Z)_i = \langle b_i, Z a_i \rangle = \langle A_i, Z \rangle_F,
\qquad A_i := b_i a_i^*,
$$
with independent properly normalized gaussian vectors $a_1,\ldots,a_m,b_1,\ldots,b_m \in \bK^n$ do not lead to the RIP \eqref{RRIP} but to the RIP \eqref{MRRIP}, while the latter still enables low-rank recovery via nuclear norm minimization, see \cite{CaiZha}.
It also enables low-rank recovery via iterative-thresholding-type algorithms, see \cite{FouSub}.
But the possibility of fulfilling this modified RIP in a regime $m \asymp n^\la r^\mu \ll n^2$
with deterministic matrices $A_1,\ldots,A_m \in \bK^{n \times n}$---of rank one or even unrestricted---is a wide open question.
This question is hereby set as a challenge to the readers.

\section{Addendum: isometric embeddings, spherical designs, tensors}
\label{SecLinks}

As mentioned in Remark \ref{RkEmb}, isometric embedding from $\ell_2^N(\bK)$ to $\ell_{2k}(\bK)$, $\bK \in \{\bR,\bC\}$,
have connections with spherical designs and tensors.
This is made precise by the following result, found in \cite{Sei} for the case $\bK = \bR$.

\begin{theorem}
\label{ThmSD}
Let $2k \ge 2$ be an even integer.
The following properties are equivalent:
\begin{itemize}
\item[1)]
There exists $A \in \bK^{N \times n}$ providing an isometric embedding from $\ell_2^n$ into $\ell_{2k}^N$, i.e., 
$$
 \|Ax\|_{2k} = \|x\|_2
\qquad \mbox{for all } x \in \bK^n;
$$
\item[2)]
There exist $x_1,\ldots,x_N \in S_2^n$ 
and $\tau_1,\ldots,\tau_N \ge 0$ with $\tau_1 + \cdots + \tau_N = 1$ such that
$$
\sum_{i,j=1}^N \tau_i \tau_j |\langle x_i , x_j \rangle|^{2k} = \int_{S_2^n \times S_2^n} |\langle x,y \rangle|^{2k} d \sigma(x) d \sigma(y),
$$
where $\sigma$ is the normalized standard measure on the unit sphere $S_2^n$ of $\ell_2^n$;
\item[3)] 
There exist $x_1,\ldots,x_N \in S_2^n$ 
and $\tau_1,\ldots,\tau_N \ge 0$ with $\tau_1 + \cdots + \tau_N = 1$ such that
$$
\sum_{i=1}^N \tau_i \otimes^k (x_i \otimes \ol{x_i}) = 
\int_{S_2^n} \otimes^k ( x \otimes \ol{x} ) d\sigma(x).
$$
\end{itemize}
\end{theorem}

Since RIPs are not genuine isometric embeddings, but almost isometric ones,
this result will be established in a slightly stronger form for the sake of completeness.
Towards this end, some pieces of notations and some identities are brought forth as a preamble.
First, one considers the quantity $\delta$
(independent of $x \in S_2^n$) defined by 
$$
\delta = \delta_{n,2k} := \int_{S_2^n} | \langle x, y \rangle |^{2k} d \sigma(y)
= \left\{ 
\bmx \df{(2k-1)\cdots 3 \cdot 1}{(n + 2k-2) \cdots (n+2) \cdot n} & \; \mbox{for }\bK = \bR,\\
\df{k!}{(n+k-1) \cdots (n+1) \cdot n} & \; \mbox{for } \bK = \bC,
\emx
\right.
$$
whose numerical value is given in \cite{Sei} for $\bK=\bR$
and in \cite{KotPev} for $\bK = \bC$.
Note that $\delta$ coincides with the double integral appearing in 2).
As for the integral appearing in~3),
called distribution $2k$-tensor (when $\bK = \bR$),
it shall be denoted by~$D$.
Thus,
$$
D = D_{n,2k} := \int_{S_2^n} \otimes^k ( y \otimes \ol{y}) d\sigma(y).
$$
Using the standard notion of inner product on tensor spaces,
one verifies below that, for any $x \in S_2^n$,
\be
\label{TensorId12}
\langle \otimes^k(x \otimes \ol{x}), D \rangle  = \delta
\qquad \quad \mbox{and} \quad \qquad
\langle D,D \rangle = \delta.
\ee
To justify the leftmost identity of \eqref{TensorId12}, it suffices to write
\begin{align*}
\langle \otimes^k(x \otimes \ol{x}), D \rangle
& = \bigg\langle \otimes^k(x \otimes \ol{x}),
\int_{S_2^n} \otimes^k ( y \otimes \ol{y}) d\sigma(y) \bigg\rangle \\
& = \int_{S_2^n} \left\langle \otimes^k(x \otimes \ol{x}), \otimes^k ( y \otimes \ol{y})  \right\rangle
d\sigma(y)\\
& = \int_{S_2^n} \left\langle x \otimes \ol{x},  y \otimes \ol{y}  \right\rangle^k d\sigma(y)
= \int_{S_2^n} \big( | \langle x ,  y  \rangle |^2 \big) ^k d\sigma(y)\\
& = \delta.
\end{align*}
To justify the rightmost identity of \eqref{TensorId12}, it suffices to write
\begin{align*}
\langle D,D \rangle
& = \bigg\langle \int_{S_2^n} \otimes^k ( x \otimes \ol{x}) d\sigma(x) , D \bigg\rangle
= \int_{S_2^n} \left\langle  \otimes^k ( x \otimes \ol{x}) , D \right\rangle d\sigma(x) = \int_{S_2^n} \delta d\sigma(x)\\ 
& = \delta.
\end{align*}
From here, an important identity follows easily, namely:
for any $x_1,\ldots,x_N \in S_2^n$ 
and any $\tau_1,\ldots,\tau_N \ge 0$ with $\tau_1 + \cdots + \tau_N = 1$,
\be
\label{TensorId3}
\bigg\| \sum_{i=1}^N \tau_i \otimes^k( x_i \otimes \ol{x_i}) - D \bigg\|_2^2
= \sum_{i,j=1}^N \tau_i \tau_j |\langle x_i, x_j \rangle|^{2k} - \delta.
\ee
Indeed, the justification of \eqref{TensorId3}
simply reads
\begin{align*}
\bigg\| & \sum_{i=1}^N  \tau_i  \otimes^k( x_i \otimes \ol{x_i}) - D \bigg\|_2^2\\
& = \sum_{i,j=1}^N \tau_i \tau_j \langle \otimes^k( x_i \otimes \ol{x_i}), \otimes^k( x_j \otimes \ol{x_j}) \rangle  -2 \Re \bigg\langle \sum_{i=1}^N \tau_i \otimes^k( x_i \otimes \ol{x_i}), D \bigg\rangle
+ \langle D, D \rangle\\
& = \sum_{i,j=1}^N \tau_i \tau_j \langle  x_i \otimes \ol{x_i},  x_j \otimes \ol{x_j} \rangle^k
- 2 \sum_{i=1}^N \tau_i \Re \left\langle \otimes^k( x_i \otimes \ol{x_i}), D \right\rangle + \langle D, D \rangle\\
& = \sum_{i,j=1}^N \tau_i \tau_j \big( | \langle  x_i,  x_j  \rangle |^2 \big)^k
- 2 \sum_{i=1}^N \tau_i \Re (\delta) + \delta,
\end{align*}
where both identities of \eqref{TensorId12} were used in the last step. 
Taking $\sum_{i=1}^N \tau_i = 1$ into account then leads to the desired identity \eqref{TensorId3}.
Note that it implies the so-called Sidelnikov inequality,
namely 
$$
\sum_{i,j=1}^N \tau_i \tau_j |\langle x_i, x_j \rangle|^{2k}
\ge \delta
\quad \mbox{whenever }
x_1,\ldots,x_N \in S_2^n, 
\; 
\tau_1,\ldots,\tau_N \ge 0, 
\; 
\sum_{i=1}^N \tau_i = 1.
$$

After this preparatory work, 
one can now state and prove the slight generalization of Theorem \ref{ThmSD},
which is retrieved as the special case $\eps_1=\eps_2=\eps_3 = 0$.

\begin{theorem}
Let $2k \ge 2$ be an even integer.
For $\eps_1,\eps_2,\eps_3 \ge 0$ that depend on each other,
the following properties are equivalent:
\begin{itemize}
\item[1')]\,
 There exists a matrix $A \in \bK^{N \times n}$ for which
$$
(1-\eps_1) \|x\|_2^{2k} \le \|Ax\|_{2k}^{2k} \le (1+\eps_1) \|x\|_2^{2k}
\qquad \mbox{for all } x \in \bK^n;
$$
\item[2')]\,
 There exist $x_1,\ldots,x_N \in S_2^n$ 
and $\tau_1,\ldots,\tau_N \ge 0$ with $\tau_1 + \cdots + \tau_N = 1$ such that
$$
\sum_{i,j=1}^N \tau_i \tau_j |\langle x_i , x_j \rangle|^{2k} \le \delta + \eps_2;
$$
\item[3')]\,
 There exist $x_1,\ldots,x_N \in S_2^n$ 
and $\tau_1,\ldots,\tau_N \ge 0$ with $\tau_1 + \cdots + \tau_N = 1$ such that
$$
\bigg\| \sum_{i=1}^N \tau_i \otimes^k (x_i \otimes \ol{x_i}) - D \bigg\|_2 \le \eps_3.
$$
\end{itemize}
\end{theorem}

\begin{proof}
2') $\Leftrightarrow$ 3') with $\eps_3 = \sqrt{\eps_2}$. 
This equivalence results from identity \eqref{TensorId3}.

3') $\im$ 1') with $\eps_1 = \eps_3/\delta$.
For any $x \in \bK^n$, one observes that the leftmost identity of  \eqref{TensorId12} yields $\langle D,\otimes^k (x \otimes \ol{x}) \rangle = \delta \|x\|_2^{2k}$ by homogeneity.
One also observes that
\begin{align*}
\bigg\langle \sum_{i=1}^N \tau_i \otimes^k (x_i \otimes \ol{x_i} ), \otimes^k (x \otimes \ol{x}) \bigg\rangle
& = \sum_{i=1}^N \tau_i \langle \otimes^k (x_i \otimes \ol{x_i} ), \otimes^k (x \otimes \ol{x}) \rangle\\
& = \sum_{i=1}^N \tau_i \big( |\langle x_i, x \rangle |^2 \big)^k.
\end{align*}
From both these observations, it follows that
\begin{align*}
\bigg| \sum_{i=1}^N \tau_i |\langle x_i,x \rangle|^{2k} - \delta \|x\|_2^{2k} \bigg|
& = \bigg|
\bigg\langle
\sum_{i=1}^N \tau_i \otimes^k (x_i \otimes \ol{x_i} ) -D , \otimes^k (x \otimes \ol{x}) \bigg\rangle \bigg| \\
& \le \bigg\| \sum_{i=1}^N \tau_i \otimes^k (x_i \otimes \ol{x_i} ) -D \bigg\|_2 
\, \Big\| \otimes^k (x \otimes \ol{x}) \Big\|_2\\
& \le \eps_3 \|x\|_2^{2k}.
\end{align*}
Therefore, setting $a_i := (\tau_i/\delta)^{\f{1}{2k}} x_i$ for each $i \in [1:N]$,
the above becomes
$$
\bigg| \sum_{i=1}^N |\langle a_i,x \rangle|^{2k} - \|x\|_2^{2k} \bigg| \le \eps_1 \|x\|_2^{2k},
$$
which reduces to 1') for the matrix $A \in \bK^{N \times n}$ with rows $a_1^*,\ldots,a_N^*$.

1') $\im$(2') with $\eps_2 = 4 \eps_1 \delta$ when $\eps_1 \le 1/2$. 
With $a_1^*,\ldots,a_N^*$ denoting the rows of the matrix $A \in \bK^{N \times n}$, 
the almost isometric embedding takes the form
\be
\label{AIE}
(1-\eps_1) \|x\|_2^{2k} 
\le \sum_{i=1}^N |\langle a_i,x \rangle|^{2k} 
\le (1+\eps_1) \|x\|_2^{2k}
\qquad \mbox{for all } x \in \bK^n.
\ee
First, integrating \eqref{AIE} over $x \in S_2^n$ implies that $1-\eps_1 \le \sum_{i=1}^N \|a_i\|_2^{2k} \delta \le 1+\eps_1$.
Setting $S : = \sum_{i=1}^N \|a_i\|_2^{2k}$,
this means that 
\be
\label{1'2'1}
\f{\delta}{1+\eps_1}  \le \f{1}{S}  \le \f{\delta}{1-\eps_1} .
\ee 
Second, setting $x_i := a_i / \|a_i\|_2 \in S_2^n$ for each $i \in [1:N]$
and selecting  $x=x_j$ in \eqref{AIE} yields
$$
1-\eps_1 \le \sum_{i=1}^N \|a_i\|_2^{2k} |\langle x_i,x_j \rangle|^{2k} \le 1 + \eps_1.
$$
Defining $\tau_i := \|a_i\|_2^{2k} / S$,
multiplying the latter by $\tau_j$ and summing over $j \in [1:N]$ leads to
\be
\label{1'2'2}
\f{1-\eps_1}{S}
\le \sum_{i,j=1}^N \tau_i \tau_j |\langle x_i, x_j \rangle|^{2k}
\le \f{1+\eps_1}{S}.
\ee
From the estimates \eqref{1'2'1} and \eqref{1'2'2}, one finally concludes that
$$
\sum_{i,j=1}^N \tau_i \tau_j |\langle x_i, x_j \rangle|^{2k} \le \f{1+\eps_1}{1-\eps_1} \delta
\le (1+4 \eps_1) \delta,
$$
where the last step used $\eps_1 \le 1/2$.
This is the desired inequality with $\eps_2 = 4 \eps_1 \delta$.
\end{proof}

\begin{acknowledgement}
S. F. is partially supported by grants from the NSF (DMS-2053172) and from the ONR (N00014-20-1-2787).
\end{acknowledgement}


\end{document}